\documentclass{amsart}

\usepackage{amssymb,color}
\usepackage{amsfonts}
\usepackage{amsmath}
\usepackage{euscript}
\usepackage{enumerate}
\usepackage{graphics}

\newtheorem{theorem}{Theorem}[section]
\newtheorem{lemma}[theorem]{Lemma}
\newtheorem{note}[theorem]{Note}
\newtheorem{prop}[theorem]{Proposition}
\newtheorem{cor}[theorem]{Corollary}

\newtheorem*{Theorem1'}{Theorem 1'}

\theoremstyle{definition}
\newtheorem{definition}[theorem]{Definition}

\theoremstyle{remark}

\numberwithin{equation}{section}

\newcommand \C{{\mathbb C}}
\newcommand \B{{\mathcal B}}

\newcommand \End{{\mathrm {End}}}
\newcommand \Sp{{\mathrm {Sp}}}
\newcommand \Un{{\mathrm {U}}}
\newcommand \Hom{{\mathrm {Hom}}}
\newcommand \dm{{\mathrm {dim}}}
\newcommand \chr{{\mathrm {char}}}

\newcommand \GL{{\mathrm {GL}}}

\newcommand \al{{\alpha}}

\newcommand \ga{{\gamma}}
\newcommand \lm{{\lambda}}

\newcommand \ind{{\mathrm {ind}}}
\newcommand \res{{\mathrm {res}}}

\newcommand \U{{\mathcal U}}

\begin{document}

\title [Representations of symplectic and unitary groups] {\small{Irreducible representations of unipotent subgroups
of symplectic and unitary groups defined over rings}}

\author{Fernando Szechtman}
\address{Department of Mathematics and Statistics, University of Regina}
\email{fernando.szechtman@gmail.com}
\thanks{The author was supported in part by an NSERC discovery grant}



\begin{abstract} Let $A$ be a ring with $1\neq 0$, not necessarily finite, endowed with an involution~$*$,
that is, an anti-automorphism of order $\leq 2$. Let $H_n(A)$ be
the additive group of all $n\times n$ hermitian matrices over $A$
relative to $*$. Let $\U_n(A)$ be the subgroup of $\GL_n(A)$ of
all upper triangular matrices with 1's along the main diagonal.
Let $P=H_n(A)\rtimes \U_n(A)$, where $\U_n(A)$ acts on $H_n(A)$ by
$*$-congruence transformations. We may view $P$ as a unipotent
subgroup of either a symplectic group $\Sp_{2n}(A)$, if $*=1_A$
(in which case $A$ is commutative), or a unitary group
$\Un_{2n}(A)$ if $*\neq 1_A$. In this paper we
construct and classify a family of irreducible representations of
$P$ over a field $F$ that is essentially arbitrary. In particular,
when $A$ is finite and $F=\C$ we obtain irreducible
representations of $P$ of the highest possible degree.

\end{abstract}

\maketitle

\section{Introduction}

Given a finite field $F_q$ of characteristic $p$, let $\U_n(q)$ stand for the Sylow $p$-subgroup of $\GL_n(q)$
consisting of all upper triangular matrices with 1's along the main diagonal. Associated to each triple $(i,j,\lm)$,
where $1\leq i<j\leq n$ and $\lm:F_q^+\to \C^*$ is a non-trivial linear character, there is an irreducible character
$\chi_{(i,j,\lm)}$ of $\U_n(q)$ of degree $q^{j-i-1}$, as constructed by Lehrer \cite{L}. In fact, Lehrer showed that
as long as
\begin{equation}
\label{sec}
i_1<i_2<\cdots<i_r,\quad j_1>j_2>\cdots>j_r
\end{equation}
and $\lm_1,\dots,\lm_r:F_q^+\to \C^*$ are non-trivial linear characters, then the product
\begin{equation}
\label{bas}
\chi_{(i_1,j_1,\lm_1)}\chi_{(i_2,j_2,\lm_2)}\cdots \chi_{(i_r,j_r,\lm_r)}
\end{equation}
is also irreducible. In particular, by taking the sequences (\ref{sec}) to be
$$
1<2<\cdots,\quad n>n-1>\cdots
$$
Lehrer obtained irreducible characters of $\U_n(q)$ of the highest possible degree, namely $q^{(n-2)+(n-4)+\cdots}$.

Lehrer's results were extended by Andr\'{e} in a series of papers begun in 1995. In \cite{A}, under the assumption that $p\geq n$, he showed
that every non-trivial irreducible character of $\U_n(q)$ is a constituent of one and only one character of the form (\ref{bas}), where none of the $i_1,\dots,i_r$ (resp. $j_1,\dots,j_r$) are repeated. Moreover, he identified amongst these {\em all} irreducible
characters of $\U_n(q)$ of the highest possible degree. The condition that $p\geq n$ was later removed in \cite{A2}.

Analogous results for Sylow $p$-subgroups of symplectic and
orthogonal groups over $F_q$ were later obtained by Andr\'{e} and
Neto \cite{AN,AN2} provided $p$ is odd. These results have been
recently generalized by  Andr\'{e}, Freitas and Neto \cite{AFN} as
well as by Andrews \cite{As} to algebra groups with involution in
the context of the supercharacter theory developed by Diaconis and
Isaacs \cite{DI}.

What concerns us here is the construction and classification of a family of irreducible modules of unipotent subgroups of symplectic and unitary groups defined over a fairly general type of ring~$A$, not necessarily finite or commutative.
The field~$F$ underlying our representations is essentially arbitrary.


It is worth noting that our modules will be finite dimensional over $F$ if and only if $A$ itself is finite. However, this fact
will play no role whatsoever in our arguments. The flexibility of being able to deal with finite and infinite dimensional modules
on equal terms is based on \cite{S}, which develops a Clifford theory for modules of arbitrary dimensionality. 

Let $A$ be a ring (all our rings are supposed to have a nonzero identity), endowed with an involution~$*$. Throughout this paper,
this means that $*$ is an anti-automorphism of $A$ of order $\leq 2$. Let $H_n(A)$ stand for the additive group of all $n\times n$ hermitian matrices, relative to $*$, over $A$, and set $P=H_n(A)\rtimes \U_n(A)$, where $\U_n(A)$ acts on $H_n(A)$ by
$*$-congruence transformations. Then $P$ is a unipotent subgroup of a unitary group $\Un_{2n}(A)$ if $*\neq 1_A$
or a symplectic group $\Sp_{2n}(A)$ if $*=1_A$, in which case $A$ is necessarily commutative. See \S\ref{int} for details.

Let $F$ be a field and let $\lm:A^+\to F^*$ be a linear character. We will refer to $\lm$ as left admissible if the kernel of the associated linear character $\lm^{\sharp}:A^+\to F^*$, given by $\lm^{\sharp}(\al)=\lm(\al+\al^*)$, contains no left ideals but $(0)$. 

Our main result can be stated as follows. A more precise formulation can be found in \S\ref{t1} and~\S\ref{t2}.

\begin{theorem}\label{maco} (a) Associated to every $1\leq i\leq n$ and every left admissible linear character $\lm: A^+\to F^*$ there is an irreducible
$P$-module $V_{i,\lm}$ of dimension $|A|^{n-i}$ over $F$.

(b) Let $D$ be any non-empty subset of $\{1,\dots,n\}$ and let $\lm$ be a choice function that assigns to each
$i\in D$ a left admissible linear character $\lm_i:A^+\to F^*$.
Then
$$
V(D,\lm)=\underset{i\in D}\bigotimes V_{i,\lm_i}
$$
is a monomial irreducible $P$-module.

(c) $V(D,\lm)\cong V(D',\lm')$ if and only if $D=D'$ and $\lm_i|_R=\lm'_i|_R$ for all $i\in D$, where $R=\{r\in A\,|\, r^*=r\}$.
\end{theorem}

We remark that if $D=\{1,\dots,n\}$, $A$ is finite and $F=\C$, then $V(D,\lm)$ is an irreducible $P$-module of the highest possible degree, namely $|A|^{n(n-1)/2}$.

The proof of Theorem \ref{maco} requires versions of Clifford's theory, Gallagher's theorem and Mackey's tensor product theorem
that work equally well for finite and infinite dimensional modules. Such tools are expounded in detail in \cite{S},
a brief summary of which appears in \S\ref{clio}. Our use of Gallagher's theorem makes the irreducibility of $V(D,\lm)$ conceptually transparent and free of calculations.

Let us briefly discuss the conditions imposed on $A$ and $F$. Assume first that $*=1_A$ and $2\in U(A)$, the unit group of $A$.
Then a linear character~$\lm:A^+\to F^*$ is left admissible if and only if $\ker\lm$ contains no left ideals of $A$ but $(0)$.
When $A$ is finite and $F=\C$ the latter condition has been extensively studied. We refer the reader to \cite{CG}, \cite{H}, \cite{La}, \cite{W} for examples and details. For the case when $A$ is not necessarily finite and $F$ need not be~$\C$ see \cite{SHI}, which studies irreducible modules of McLain groups defined over rings admitting such linear characters. In terms of applications, perhaps the prime example is the following. Let $K$ be a non-archimidean local field with ring of
integers~${\mathcal O}$, maximal ideal ${\mathfrak p}$ and residue
field $F_q={\mathcal O}/{\mathfrak p}$ of odd characteristic
$p$. Then $A={\mathcal O}/{\mathfrak p}^m$ is a finite, principal,
local, commutative ring of size $q^m$ affording an admissible linear character $A^+\to F^*$
when $F$ has a root of unity of order $p^m$ if
$\chr(K)=0$ and $p$ if $\chr(K)=p$. Suppose next $*\neq 1_A$. We show in \S\ref{t2} that if $B$ is a local ring admitting
a linear character~$B^+\to F^*$ whose kernel contains no left ideals of $B$ but $(0)$ (there is an abundance of such $B$ as indicated
above) one can easily construct a quadratic extension $A$ of $B$ with an involution of order 2 and a left admissible linear character $A^+\to F^*$. An important special case is $B={\mathcal O}/{\mathfrak p}^m$, in the above notation, where $p$ is now allowed to be even.

\section{Clifford theory}\label{clio}

We fix a field $F$ for the remainder of the paper.
Let $N\unlhd G$ be groups and let $W$ be an $N$-module. For $g\in
G$, consider the $N$-module~$W^g$, whose underlying $F$-vector
space is $W$, acted upon by $N$ as follows:
$$
x\cdot w= (g x g^{-1})w,\quad  x\in N,w\in W.
$$
Then
$$
I_G(W)=\{g\in G\,|\, W^g\cong W\}
$$
is a subgroup of $G$, containing $N$, called the inertia group of $W$ (cf.
\cite[Lemma 3.1]{S}.

For $G$-modules $X$ and $Y$ we define
$$
(X,Y)_G=\dm_F\Hom_{G}(X,Y).
$$

\begin{theorem}\label{cl2} Let $N\unlhd G$ be groups and
let $W$ be an irreducible $N$-module with inertia group $T$.
Suppose $S$ is an irreducible $T$-module lying over $W$. Then
$\ind_T^G S$ is an irreducible $G$-module.

In particular, if $I_G(W)=N$ then $V=\ind_N^G W$ is irreducible
and if, in addition, $(W,W)_N=1$ then $(V,V)_G=1$ as well.
\end{theorem}

\begin{proof} See \cite[Theorem 3.5]{S} for the first assertion. As for
the second, by Frobenius reciprocity (cf. \cite[Theorem 5.3]{S}),
$\End_N(W)=\Hom_N(W,V)\cong_F\End_G(V)$.
\end{proof}

\begin{lemma}\label{formol} Let $N\unlhd G$ be groups and let $W$ be a $G$-module
with $\res_N^G W$ irreducible and $(W,W)_N=1$. Let $U_1,U_2$ be
$G$-modules acted upon trivially by $N$ and suppose
$T\in\Hom_G(U_1\otimes W,U_2\otimes W)$. Then $T=S\otimes 1$,
where $S\in\Hom_{G}(U_1,U_2)$.
\end{lemma}

\begin{proof} Since $T\in\Hom_N(U_1\otimes W,U_2\otimes W)$, \cite[Lemma 3.7]{S}
implies $T=S\otimes 1$, where $S\in\Hom_F(U_1,U_2)$. But
$T\in\Hom_G(U_1\otimes W,U_2\otimes W)$, so
$S\in\Hom_{G}(U_1,U_2)$.
\end{proof}

\begin{theorem}\label{cli4}  Let $N\unlhd G$ be groups and let $W$ be a
$G$-module with $\res_N^G W$ irreducible and $(W,W)_N=1$. Let $U$
be an irreducible $G$-module acted upon trivially by $N$. Then
$U\otimes W$ is an irreducible $G$-module.
\end{theorem}

\begin{proof} This can be found in \cite[Theorem 3.11]{S}.
\end{proof}

\section{Symplectic and unitary groups over rings}\label{int}

Let $A$ be a ring, not necessarily finite, endowed with an involution~$*$. Given any $m\geq 1$, we can extend $*$ to an involution,
also denoted by $*$, of $M_{m}(A)$ by declaring:
$$
(B^*)_{ij}=(B_{ji})^*,\quad B\in M_m(A).
$$
Let $V$ be a right $A$-module and let $h:V\times V\to A$ be a skew-hermitian form relative to $*$. This means that $h$ is linear in
the second variable and satisfies:
$$
h(v,u)=-h(u,v)^*,\quad u,v\in  V.
$$
We will further assume that $V$ is free with basis $\B=\{u_1,\dots,u_n,v_1,\dots,v_n\}$ and that the Gram matrix of $h$ relative to $\B$ is
$$
J=\left(
   \begin{array}{cc}
     0 & 1_n \\
     -1_n & 0 \\
   \end{array}
 \right).
$$
Let $U$ stand for the subgroup of $\GL(V_A)$ preserving $h$. (Thus $U$ is a symplectic or unitary group depending on the order of $*$.)

Let $g\in \GL(V_A)$ and let $X=M_\B(g)\in\GL_{2n}(A)$ be the matrix of $g$ relative to~$\B$. Then $g\in U$
if and only if
\begin{equation}
\label{tobe}
X^* J X=J.
\end{equation}
For $w_1,\dots,w_m\in V$, let $\langle w_1,\dots,w_m\rangle$ stand for the $R$-span of $w_1,\dots,w_m$ in $V$. We extend
this notation so that $\langle \emptyset\rangle=(0)$. The use of $\langle \emptyset\rangle$ will be frequent but implicit.

We set
$$
M=\langle u_1,\dots,u_n\rangle,\quad N=\langle v_1,\dots,v_n\rangle.
$$
Let $C$ be the pointwise stabilizer of $M$ in $U$. Then, by (\ref{tobe}), $g\in C$
if and only if
\begin{equation}
\label{forma}
X=\left(
   \begin{array}{cc}
     1_n & S \\
     0 & 1_n \\
   \end{array}
 \right),
\end{equation}
where $S$ belongs to $H_n(A)$, the additive group of all $n\times n$ hermitian matrices, relative to $*$, over $A$.
It is clear that $C\cong H_n(A)$. Let $T$ stand for the global stabilizer of $M$ and $N$. According to (\ref{tobe}), $g\in T$
if and only if
\begin{equation}
\label{xy}
X=\left(
   \begin{array}{cc}
     Y & 0 \\
     0 & (Y^{-1})^* \\
   \end{array}
 \right),
\end{equation}
where $Y\in\GL_n(A)$. Thus $T\cong\GL_n(A)$. It is clear that $C$ and $T$ intersect trivially and $T$ normalizes $C$.
In fact,
\begin{equation}
\label{forma2}
\left(
   \begin{array}{cc}
     Y & 0 \\
     0 & (Y^{-1})^* \\
   \end{array}
 \right)\left(
   \begin{array}{cc}
     1_n & S \\
     0 & 1_n \\
   \end{array}
 \right)\left(
   \begin{array}{cc}
     Y^{-1} & 0 \\
     0 & Y^* \\
   \end{array}
 \right)=\left(
   \begin{array}{cc}
     1_n & YSY^* \\
     0 & 1_n \\
   \end{array}
 \right),
\end{equation}
so that $C\rtimes T\cong H_n(A)\rtimes \GL_n(A)$, with $\GL_n(A)$ acting on $H_n(A)$ by congruence transformations, that is,
$Y\cdot S=YSY^*$.

Let us write $\U_n(A)$ for the upper unitriangular group, that is, the
subgroup of $\GL_n(A)$ of all upper triangular matrices with 1's along the main diagonal.
Let $L$ stand for the subgroup of $T$ of all $g\in T$ satisfying
$$
gu_i\equiv u_i\mod\langle u_1,\dots,u_{i-1}\rangle,\quad 1\leq i\leq n.
$$
Thus, $g\in T$ is in $L$ if and only if $Y\in \U_n(A)$ in (\ref{xy}) (or, equivalently, $(Y^{-1})^*$ is
lower unitriangular in (\ref{xy})).

We are concerned with certain irreducible representations of the unipotent group
$$
P=C\rtimes L.
$$

For $0\leq i\leq n$, let $H_i$ be the subgroup of all $S\in H_n(A)$ whose first $i$ rows/columns
are arbitrary and whose remaining lower right $(n-i)\times (n-i)$ block is equal to 0.

For $0\leq i\leq n$, we set
$$
C_i=\{g\in C\,|\, gw\equiv w\mod\langle u_1,\dots,u_i\rangle\text{ for all }w\in \langle v_{i+1},\dots,v_n\rangle\}.
$$
Then a given $g\in C$ is in $C_i$ if and only if $S$, as in (\ref{forma}), is in $H_i$. Thus
$$1=C_0\subset C_1\subset\cdots\subset C_n=C$$
are subgroups of $C$.

\begin{lemma}\label{nors}
Each $C_i$ is a normal subgroup of $P$.
\end{lemma}

\begin{proof} Since $C$ is abelian, it suffices to show that $L$ normalizes $C_i$. In view of (\ref{forma2}), this
is equivalent to $YSY^*\in H_i$ for every $S\in H_i$ and $Y\in\U_n(A)$. Since $\U_n(A)$ is generated by all
$1+ae_{k\ell}$, where $a\in A$ and $k<\ell$, we are reduced to showing that, given $S\in H_i$, we have
$$
T=(1+ae_{k\ell})S(1+a^*e_{\ell k})\in H_i,\quad a\in A, k<\ell.
$$
Now, $(1+ae_{k\ell})S$ is obtained from $S$ by adding a multiple
of row $\ell$ to row $k$, and $T$ is obtained from
$(1+ae_{k\ell})S$ by adding a multiple of column $\ell$ to column
$k$. Thus, the lower right $(n-i)\times (n-i)$ block of $T$ is
still equal to 0. Since $T$ is hermitian, it follows that $T\in
H_i$.
\end{proof}

For $0\leq i\leq n$, we set
$$
L_i=\{g\in L\,|\, gv_{i+1}=v_{i+1},\dots,gv_n=v_n\}=\{g\in L\,|\, gM\subseteq \langle u_{1},\dots,u_i\rangle\}.
$$
Then a given $g\in L$ is in $L_i$ if and only if $Y\in \U_n(A)$, as in (\ref{xy}), has arbitrary $i$ first rows and
the remaining $n-i$ rows are those of $I_n$. Clearly,
$$1=L_0\subset L_1\subset\cdots\subset L_n=L$$
are subgroups of $L$.

\begin{lemma}\label{nors2}
Each $L_i$ is a normal subgroup of $L$.
\end{lemma}

\begin{proof} If $g\in L$, $h\in L_i$ then $
ghg^{-1}M\subseteq ghM\subseteq g\langle u_{1},\dots,u_i\rangle\subseteq \langle u_{1},\dots,u_i\rangle
$.
\end{proof}

For $1\leq i\leq n$ and $0\leq j\leq i$, we set
$$
N_{ij}=C_i\rtimes L_j.
$$

\begin{lemma}\label{nors3}
Each $N_{ij}$ is a normal subgroup of $P$.
\end{lemma}

\begin{proof} In view of Lemmas \ref{nors} and \ref{nors2}, it suffices to show that $[C,L_j]\subseteq C_j$.
For this purpose, note that
\begin{equation}
\label{fort}
S-(1+ae_{k\ell})S(1+a^*e_{\ell k})\in H_j,\quad a\in A, k<\ell\leq j, S\in H_n(A).
\end{equation}
This is clear, since $S$ and $(1+ae_{k\ell})S(1+a^*e_{\ell k})$
have identical lower right blocks of size $(n-j)\times (n-j)$. It
follows from (\ref{forma2}) and (\ref{fort}) that if $x\in C$ and
$y\in L_j$ is represented by (\ref{xy}) with $Y=1+ae_{k\ell}$,
where $a\in A$ and $k<\ell\leq j$, then
\begin{equation}
\label{fort2} [x,y]\in C_j.
\end{equation}
On the other hand, the matrices representing $L_j$ are generated by all such $Y$,
$$
[x,yz]=[x,y]y[x,z]y^{-1},
$$
and $C_j$ is normal in $P$ by Lemma \ref{nors}. This and (\ref{fort2}) yield $[C,L_j]\subseteq C_j$.
\end{proof}

For $1\leq i\leq n$, we set
$$
N_i=C_i\rtimes L_i,\quad N_i^0=C_i\rtimes L_{i-1}.
$$
\begin{cor}\label{nors4}
Each $N_{i}$ and $N_{i}^0$ is a normal subgroup of $P$.
\end{cor}

\section{Irreducible modules arising from Clifford's theorem}\label{t1}

Let
$$
R=\{r\in R\,|\, r^*=r\}.
$$
This is a subgroup of $A^+$. If $R$ is central in $A$, then $R$ is also a subring of $A$.

\begin{lemma}\label{ula} Let $v\in N$. Then
the map $\delta_v:C\to A$, given by,
$$
\delta_v(x)=h(xv,v),\quad x\in C,
$$
is a group homomorphism satisfying $\delta_v(x)\in R$ for all
$x\in C$.
\end{lemma}

\begin{proof} We first verify that $\delta_v$ takes values in $R$ and not just in $A$. Indeed, given $x\in C$
we have $xv=v+u$, where $u\in M$, so $v=x^{-1}v+u$, whence $x^{-1}v=v-u$. Moreover, since $h(v,v)=0$, we have
$$
\begin{aligned}
h(xv,v)^* &=-h(v,xv)
=-h(v,v+u)
=-h(v,u)
=h(v,v-u)
=h(v,x^{-1}v)\\
&=h(xv,v).
\end{aligned}
$$
This shows $h(xv,v)\in R$. Next we show that $\delta_v$ is a group homomorphism. Indeed, suppose $y\in C$. Then
$yv=v+w$, where $w\in M$, so
$$
h(xyv,v)=h(x(v+w),v)=h(v+u+w,v)=h(xv,v)+h(yv,v).
$$
\end{proof}

\begin{cor}\label{nors5} Let $\lm:A^+\to F^*$ be a linear character and let $v\in N$. Then
the map $\chi_{v,\lm}:C\to F^*$, given by,
\begin{equation}
\label{ch}
\chi_{v,\lm}(x)=\lm(h(xv,v)),\quad x\in C,
\end{equation}
is a group homomorphism.
\end{cor}

Note that given $v\in N$, a linear character $\lm:A^+\to F^*$, and $t\in T$, we have
\begin{equation}
\label{keko} ({}^t
\chi_{v,\lm})(x)=\chi_{v,\lm}(x^t)=\chi_{v,\lm}(t^{-1}x
t)=\chi_{tv,\lm}(x),\quad x\in C.
\end{equation}

Let $L_v$ stand for the pointwise stabilizer of $v$ in $L$. By
(\ref{keko}), $L_v$ is contained in the inertia group of
$\chi_{v,\lm}$ in $P$ and $\chi_{v,\lm}$ can be extended to a
linear character of $C\rtimes L_v$, also denoted by
$\chi_{v,\lm}$, by means of the same formula (\ref{ch}). Note that
for $1\leq i\leq n$, $N_i^0$ is contained in $C\rtimes L_{v_i}$,
so $\chi_{v_i,\lm}$ is defined on~$N_i^0$.

\begin{definition}
A linear character $\lm:A^+\to F^*$ is
said to be left admissible if the kernel of the associated linear character $\lm^{\sharp}:A^+\to F^*$,
given by $\lm^{\sharp}(a)\mapsto \lm(a+a^*)$, does not contain any left ideals
in its kernel but $(0)$.
\end{definition}

\begin{prop}\label{po} Suppose $\lm:A^+\to F^*$ is a left admissible linear character and let $1\leq i\leq n$. Then the inertia
group of $\chi_{v_i,\lm}|_{C_i}$ in $N_i$ is $N_i^0$.
\end{prop}

\begin{proof} Let $B_i$ be the subgroup of $L_i$ of all $g\in L_i$ such that $gv_j=v_j$ for all $j\neq i$. Then
$$
N_i=N_i^0\rtimes B_i.
$$
Suppose $b\in B_i$ stabilizes $\chi_{v_i,\lm}|_{C_i}$. We wish to prove that $b=1$. We have
$$
bv_i=v_i+v_{i+1}c_{i+1}+\cdots+v_n c_n,\quad c_j\in A.
$$
Suppose $i<j\leq n$. We next show that $c_j=0$. Indeed, for $a\in A$ let $g_a\in C_i$ be defined by
$$
g_av_i=v_i+u_j a^*,\; g_av_j=v_j+u_i a,\; g_av_k=v_k\text{ if }k\neq i,j.
$$
Then, on the one hand
$$
\chi_{v_i,\lm}(g_a)=\lm(h(g_av_i,v_i))=\lm(h(v_i+u_j a^*,v_i))=1,
$$
while on the other hand
$$
\begin{aligned}
{}^b \chi_{v_i,\lm}(g_a) &= \lm(h(g_abv_i,bv_i))\\
&= \lm(h(g_a(v_i+v_{i+1}c_{i+1}+\cdots+v_n c_n),bv_i))\\
&=\lm(h(v_i+u_ja^*+v_{i+1}c_{i+1}+\cdots+v_jc_j+u_iac_j+\cdots+v_n c_n,bv_i)\\
&=\lm(ac_j+(ac_j)^*).
\end{aligned}
$$
The left admissibility of $\lm$ and
$$
\lm(ac_j+(ac_j)^*)=1,\quad a\in A,
$$
force $c_j=0$.
\end{proof}

For $1\leq i\leq n$, let $W_i=F w_i$ be a 1-dimensional $N_i^0$-module upon which $N_i^0$ acts via $\chi_{v_i,\lm}|_{N_i^0}$.

\begin{theorem}\label{irrn} Suppose $\lm:A^+\to F^*$ is a left admissible linear character and let $1\leq i\leq n$. Then
$$
U_{i,\lm}=\ind_{N_i^0}^{N_i} W_i
$$
is an irreducible $N_i$-module satisfying $(U_{i,\lm},U_{i,\lm})_{N_i}=1$.
\end{theorem}

\begin{proof} This follows from Theorem \ref{cl2}, Corollary \ref{nors4} and Proposition \ref{po}.
\end{proof}

\begin{cor}\label{po2} Suppose $\lm:A^+\to F^*$ is a left admissible linear character and let $1\leq i\leq n$. Then the inertia
group of $\chi_{v_i,\lm}|_C$ in $P$ is $C\rtimes L_{v_i}$.
\end{cor}

\begin{proof} Immediate consequence of Proposition \ref{po}.
\end{proof}

\begin{theorem}\label{irrg} Suppose $\lm:A^+\to F^*$ is a left admissible linear character and let $1\leq i\leq n$.
Extend the action of $N_i^0$ on $W_i$ to $C\rtimes L_{v_i}$ by letting $C\rtimes L_{v_i}$ act on $W_i$ via $\chi_{v_i,\lm}$. Then
$$
V_{i,\lm}=\ind_{C\rtimes L_{v_i}}^{P} W_i
$$
is an irreducible $P$-module satisfying $(V_{i,\lm},V_{i,\lm})_{P}=1$.
\end{theorem}

\begin{proof} Immediate consequence of Theorem \ref{cl2}, Corollary \ref{nors4} and Corollary \ref{po2}.
\end{proof}

\begin{prop}\label{restri} Suppose $\lm:A^+\to F^*$ is a left admissible linear character and let $1\leq i\leq n$. Then
$$
\res^P_{N_i} V_{i,\lm}\cong U_{i,\lm}.
$$
\end{prop}

\begin{proof} We readily see that a set of representatives for the cosets of $N_i^0$ in $N_i$ (e.g.
the abelian group $B_i$ used in the proof of Proposition \ref{po}) is also
a set of representatives for the left cosets of $C\rtimes L_{v_i}$ in $P$. The result follows.
\end{proof}

\begin{note}\label{degr}
{\rm The above modules have dimension $|A|^{n-i}$ (whether $A$ is
finite or not). }
\end{note}

\begin{prop}\label{uy} Suppose $\lm:A^+\to F^*$ is a left admissible linear character and $1\leq i<j\leq n$.
Then $N_i$ acts trivially on $V_{j,\lm}$.
\end{prop}

\begin{proof} We know from Corollary \ref{nors4} that $N_i$ is normal in $P$ and from Theorem \ref{irrg}
that $V_{j,\lm}$ is an irreducible $P$-module. Thus, it suffices
to show that $N_i$ has a nonzero fixed point in $V_{j,\lm}$. On
the other hand, by Proposition \ref{restri}, $U_{j,\lm}$ is an
$N_j$-submodule of $V_{j,\lm}$. Since $N_i$ is contained in $N_j$,
it is enough to prove that $N_i$ has a nonzero fixed point in
$U_{j,\lm}$. From the definition of $U_{j,\lm}$ given in Theorem
\ref{irrn}, the existence of a nonzero fixed point for $N_i$ in
$U_{j,\lm}$ is ensured provided
\begin{equation}
\label{xg}
\lm(h(gv_j,v_j))=1,\quad g\in N_i.
\end{equation}
Let $g\in N_i$. Then $g=xy$, where $x\in C_i$ and $y\in L_i$. Obviously, $yv_j=v_j$, while
$xv_j=v_j+z$, with $z\in\langle u_1,\dots,u_i\rangle$, so (\ref{xg}) follows.
\end{proof}

\section{Irreducible modules arising from Gallagher's theorem}\label{t2}

\begin{theorem}\label{51} Let $D$ be any non-empty subset of $\{1,\dots,n\}$ and let $\lm$ be a choice function such that for any $i\in D$,
$\lm_i$ is a left admissible linear character $A^+\to F^*$. Then
$$
V(D,\lm)=\underset{i\in D}\bigotimes V_{i,\lm_i}
$$
is an irreducible $P$-module satisfying $(V(D,\lm),V(D,\lm))_P=1$.
\end{theorem}

\begin{proof} By induction on $|D|$. The case $|D|=1$ is proven in Theorem~\ref{irrg}. Suppose $|D|>1$
and the result is true for subsets of $\{1,\dots,n\}$ of
size~$<|D|$. Let $i$ be the smallest element of $D$ and set
$E=D\setminus\{i\}$. By Corollary \ref{nors4}, $N=N_i$ is normal
in $G=P$. Moreover, by Theorem~\ref{irrn} and Proposition
\ref{restri}, if $W=V_{i,\lm_i}$, then $\res_N^G W$ is irreducible
and $(W,W)_N=1$. By inductive hypothesis,
$$
U=\underset{j\in E}\bigotimes V_{j,\lm_j}
$$
is an irreducible $G$-module, which is acted upon trivially by
$N$, due to Proposition \ref{uy}. It follows from
Theorem~\ref{cli4} that $V(D,\lm)$ is an irreducible $G$-module.
Moreover, Lemma \ref{formol} implies that
$(V(D,\lm),V(D,\lm))_G=1$.
\end{proof}

\begin{theorem}\label{monom} Let $D$ be any non-empty subset of $\{1,\dots,n\}$ and let $\lm$ be a choice function such that for any $i\in D$,
$\lm_i$ is a left admissible linear character $A^+\to F^*$. Let $L_D$ be the pointwise stabilizer of $\{v_i\,|\, i\in D\}$ in $L$
and let
$$
M_D=C\rtimes L_D=C\rtimes \underset{i\in D}\bigcap L_{v_i}=\underset{i\in D}\bigcap (C\rtimes L_{v_i}).
$$
Let $\chi_{D,\lm}:M_D\to F^*$ be the linear character defined by
$$
\chi_{D,\lm}(g)=\underset{i\in D}\prod \chi_{v_i,\lm_i}(g)=\underset{i\in D}\prod \lm_i(h(gv_i,v_i)).
$$
Let $Z_D$ be a 1-dimensional $M_D$-module upon which $M_D$ acts via $\chi_D$. Then
$$
V(D,\lm)\cong \ind_{M_D}^P Z_D.
$$
\end{theorem}

\begin{proof} By induction on $|D|$. If $D|=1$ there is nothing to do. Suppose $|D|>1$
and the result is true for subsets of $\{1,\dots,n\}$ of size~$<|D|$. Let $i$ be any element of $D$
and set $E=D\setminus\{i\}$. Let $\mu$ be the restriction of $\lm$ to $E$ and set
$$
V(E,\mu)=\underset{j\in E}\bigotimes V_{j,\lm_j}.
$$
Then
$$
V(D,\lm)=V_{i,\lm_i}\otimes V(E,\mu).
$$
By inductive hypothesis, $V(E,\mu)\cong \ind_{M_E}^P T$, where $T=Ft$ is acted upon $M_E$ via $\chi_{E,\mu}$.
Since $L_EL_{v_i}=L$ and $L_E\cap L_{v_i}=L_D$, we obtain $M_EM_{\{i\}}=P$ and $M_E\cap M_{\{i\}}=M_D$, so
Mackey Tensor Product Theorem (cf. \cite[Theorem 2.1]{S}) yields $V(D,\lm)\cong\ind_{M_D}^P Z_D$.
\end{proof}

\begin{theorem}\label{53} Let $D$ (resp. $D'$) be any non-empty subset of $\{1,\dots,n\}$ and let $\lm$ (resp. $\lm'$) be a choice function such that for any $i\in D$ (resp. $i\in D'$),
$\lm_i$ (resp. $\lm'_i$) is a left admissible linear character $A^+\to F^*$. Suppose the map $\Lambda:A\to R$ given by $a\mapsto a+a^*$
is surjective. Then
$$
V(D,\lm)\cong V(D',\lm')\Leftrightarrow D=D'\text{ and }\lm_i|_R=\lm'_i|_R\text{ for all }i\in D.
$$
\end{theorem}

\begin{proof} Sufficiency is consequence of Lemma \ref{ula}. As for necessity, assume $V(D,\lm)\cong V(D',\lm')$. Let $i$ (resp. $i'$)
be the smallest element of $D$ (resp. $D'$).
Suppose, if possible, that $i<i'$. Then, by Proposition \ref{uy}, $N_i^0$ acts trivially on $V(D',\lm')$, while
$V(D,\lm)$ has a 1-dimensional $N_i^0$-submodule upon which $N_i^0$ acts via $\chi_{v_i,\lm_i}$. For $a\in A$, let $g_a\in C_i\subset N_i^0$
be defined by
$$
g_av_i=v_i+u_i(a+a^*),\; g_av_j=v_j, j\neq i.
$$
Then
\begin{equation}
\label{chot}
\chi_{v_i,\lm_i}(g_a)=\lm_i(h(g_av_i,v_i))=\lm_i(a+a^*).
\end{equation}
Since $\lm_i$ is left admissible, there is $a\in A$ such that $\chi_{v_i,\lm_i}(g_a)\neq 1$. This contradicts $V(D,\lm)\cong V(D',\lm')$.
Likewise we see the impossibility of $i'<i$. This shows $i=i'$.

It follows from Propositions \ref{restri} and \ref{uy} that $V(D,\lm)$ (resp. $V(D',\lm')$) is the direct sum of 1-dimensional $N_i^0$-submodules
upon which $N_i^0$ acts via $\chi_{v_i,\lm_i}$ (resp. $\chi_{v_i,\lm'_i}$) and its $B_i$-conjugates, where $B_i$ is as in
the proof of Propositions \ref{po}. The argument given in that proof shows that the only $B_i$-conjugate of $\chi_{v_i,\lm_i}$
that can possibly equal $\chi_{v_i,\lm'_i}$ is $\chi_{v_i,\lm_i}$ itself, so $\chi_{v_i,\lm_i}=\chi_{v_i,\lm'_i}$.
This and (\ref{chot}) imply that $\lm_i$ and $\lm'_i$ agree on the image of $\Lambda$. By hypothesis this image is $R$, so
$\lm_i|_R=\lm'_i|_R$.

Let $E$ and $\mu$ (resp. $E'$ and $\mu'$) be defined as in the proof of Theorem \ref{monom}. Then
$$
V(D,\lm)=V(E,\mu)\otimes V_{i,\lm_i},\quad
V(D',\lm')=V(E',\mu')\otimes V_{i,\lm_i},
$$
where $V(E,\mu)$ (resp. $V(E',\mu')$) is understood to be the
trivial $P$-module if $|D|=1$ (resp. $|D'|=1$). It follows from
Lemma \ref{formol}, Theorem \ref{irrn}, Proposition \ref{restri}
and Proposition \ref{uy} that $V(E,\mu)\cong V(E',\mu')$. The
above argument makes it clear that $|D|=1$ if and only if
$|D'|=1$. The result now follows by induction.
\end{proof}

\begin{note}{\rm By Note \ref{degr} and Theorem \ref{51}, if we take $D=\{1,\dots,n\}$ and vary the choice
function $\lm$ according to Theorem
\ref{53} we obtain distinct non-isomorphic irreducible
$P$-modules of degree $|A|^{n(n-1)/2}$. This is precisely the
index in $P$ of the normal abelian subgroup $C$. Thus, when $A$ is
finite, it follows from Ito's theorem that
$|A|^{n(n-1)/2}$ is the highest possible degree of an irreducible
complex $P$-module. }
\end{note}

\begin{note}{\rm Irreducible modules of Sylow $p$-subgroups of $\Sp_{2n}(q)$ and $\mathrm{U}_{2n}(q^2)$, where $q$ has characteristic $p$,
where also studied in \cite{Sz}. Even in this special case, the proofs in \cite{Sz} are more computational and sometimes incomplete,
specially that of \cite[Theorem 4.1]{Sz}, where only the case $D=\{1,\dots,n\}$ is considered.}
\end{note}

Given a ring $B$, we let $U(B)$ and $J(B)$ stand for the unit group and Jacobson radical of $B$. Note that if $B$ is local then either $2\in U(B)$ or $2\in J(B)$. It is clear that if $*=1_A$ and $2\in U(A)$ then $\Lambda$, as defined in Theorem \ref{53}, is surjective. For the case
$*\neq 1_A$ see Lemma \ref{quad} below.

\begin{lemma}\label{quad} Let $B$ be a ring having a left primitive linear character $\mu:B^+\to F^*$. Suppose that
either $2\in U(B)$ or $2\in J(B)$ and let $f(t)=t^2+ct+d\in B[t]$ be an arbitrary polynomial subject only to:

$\bullet$ $c$ and $d$ are central in $B$;

$\bullet$ $c=0$ and $d\in U(B)$ if $2\in U(B)$.

$\bullet$ $c\in U(B)$ if $2\in J(B)$.

Let $A=B[t]/(f(t))=B[\ga]$, where $\ga=t+(f(t))$. Then

(a) $A$ is a free (left and right) $B$-module with basis $\{1,\ga\}$.

(b) There is a unique involution $*$ of $A$, of order 2, extending $1_B$ and satisfying $$\ga^*=-(c+\ga).$$

(c) The linear character $\lm:A^+\to F^*$, given by
$$
\lm(a+b\ga)=\mu(a),\quad a,b\in B,
$$
is left admissible.

(d) The map $\Lambda:A\to B$, given by $\alpha\mapsto  \al+\al^*$, is surjective.
\end{lemma}

\begin{proof} (a) and (b) are routinely verified. Let us confirm (c). Suppose $I$ is a left ideal of $A$ contained in the kernel of $\lm^{\sharp}$
and let $\al=a+b\ga$, where $a,b\in B$, be an arbitrary element of $I$. We wish to show that $\al=0$. Note that
\begin{equation}
\label{qq}
\al+\al^*=2a-cb\in B.
\end{equation}
For any $e\in B$, we have $e\al\in I$, so that
$$
0=\lm^{\sharp}(e\al)=\lm(e\al+(e\al)^*)=\lm (e(\al+\al^*))=\lm (e(2a-cb))=\mu (e(2a-cb)).
$$
Thus $B(2a-cb)\subseteq\ker\mu$, whence
\begin{equation}\label{2a}
2a=cb
\end{equation}
by the left primitivity of $\mu$. Now $-\ga\al\in I$, where
$$
-\ga\al=-(a\ga+b\ga^2)=-(a\ga+b(-c\ga-d))=db+(cb-a)\ga.
$$
We infer from (\ref{2a}) that
\begin{equation}\label{3a}
2db=c^2b-ca.
\end{equation}
Assume first that $2\in U(B)$. In this case $c=0$ and $d\in U(B)$, so (\ref{2a}) and (\ref{3a}) give $a=0$ and $b=0$, so $\al=0$.
Assume next $2\in J(B)$. In this case $c\in U(B)$. Substituting (\ref{2a}) in (\ref{3a}) gives
\begin{equation}\label{4a}
2db=2ca-ca=ca.
\end{equation}
Multiplying (\ref{4a}) through by 2 and using (\ref{2a}) once more yields
$$
4db=c^2b,
$$
so
$$
(4d-c^2)b=0.
$$
Since $2\in J(B)$ and $c\in U(B)$, it follows that $4d-c^2\in U(B)$, so $b=0$. Substituting this in (\ref{3a}) gives $ca=0$, whence $a=0$
and therefore $\al=0$. This proves (c), while (d) follows easily from (\ref{qq}) and the stated conditions on $c$.
\end{proof}

A special case of Lemma \ref{quad}, applied to find a formula for the Weil character of unitary groups over
finite, principal, local, commutative rings of odd characteristic, can be found in \cite{GS}.




\end{document}